\documentclass[12pt,letterpaper]{article}
\usepackage{graphicx}
\usepackage{amsmath}
\usepackage{amsfonts}
\usepackage{amsthm}
\usepackage{amssymb}
\usepackage{color}

\usepackage{verbatim} 
\usepackage{hyperref}
\usepackage{breakurl}

\newtheorem{theorem}{Theorem}[section]
\newtheorem{lemma}[theorem]{Lemma}
\newtheorem{corollary}[theorem]{Corollary}

\newcommand{\be}{\begin{equation}}
\newcommand{\ee}{\end{equation}}
\newcommand{\bea}{\begin{eqnarray}}
\newcommand{\eea}{\end{eqnarray}}
\newcommand{\beas}{\begin{eqnarray*}}
\newcommand{\eeas}{\end{eqnarray*}}

\begin{document}

\title{Baron M\"{u}nchhausen's Sequence}

\author{Tanya Khovanova\\MIT \and Konstantin Knop\\Youth Math School of St.Petersburg State University \and Alexey Radul\\MIT}
\maketitle

\begin{abstract}
We investigate a coin-weighing puzzle that appeared in the all-Russian
math Olympiad in 2000.  We liked the puzzle because the methods of
analysis differ from classical coin-weighing puzzles.  We generalize
the puzzle by varying the number of participating coins, and deduce a
complete solution---perhaps surprisingly, the objective can be
achieved in no more than two weighings regardless of the number of
coins involved.
\end{abstract}

\section{Introduction}

The following coin-weighing puzzle, due to Alexander Shapovalov,
appeared in the Regional round of the all-Russian math Olympiad in
2000 \cite{BaronOriginal}.

\begin{quote}
Eight coins weighing $1, 2,\ldots, 8$ grams are given, but which
weighs how much is unknown.  Baron M\"{u}nchhausen claims he knows
which coin is which; and offers to prove himself right by conducting
one weighing on a balance scale, so as to unequivocally demonstrate
the weight of at least one of the coins.  Is this possible, or is he
exaggerating?
\end{quote}

We chose to investigate this puzzle partly because classical
coin-weighing puzzles \cite{Guy} tend to ask the person designing the
weighings to discover something they do not know, whereas here the
party designing the weighings knows everything, and is trying to use
the balance scale to convince someone else of something.  There is
therefore a difference of method: when the weigher is an investigator,
they typically find themselves playing a minimax game against fate:
they must construct experiments all of whose possible outcomes
are as equally likely as possible, in order to learn the most they
can even in the worst case.  The Baron, however, knows everything,
so he has the liberty to construct weighings whose results look
very surprising from the audience's perspective.

We invite the reader to experience the enjoyment of solving this
puzzle for themselves before proceeding; we will spoil it completely
on page~\pageref{puzzle-solution}.

\subsection{The Sequence}

We will generalize this puzzle to $n$ coins that weigh $1, 2,\ldots,
n$ grams.  We are interested in the minimum number of weighings on a
balance scale that the Baron needs in order to convince his audience
about the weight of at least one of those coins.  It turns out that
the answer is never more than two; over the course of the paper, we will
prove this, and determine in closed from which $n$ require two
weighings, and which can be done in just one.

\subsection{The Roadmap}

In Section~\ref{sec:sequence} we define the Baron's sequence again and
show some of the flavor of this problem by calculating the first few
terms.  In Section~\ref{sec:three} we prove the easy but perhaps
surprising observation that this sequence is bounded; in fact that no
term of this sequence can exceed three.  That theorem opens the door to a
complete description of all the terms of Baron M\"{u}nchhausen's
sequence, which we begin in Section~\ref{sec:ones} by explicitly
finding the terms that are equal to one \cite{blog2}, to wit, the numbers $n$ of
coins such that the Baron can prove the weight of one coin among $n$
in just one weighing.

Discriminating between two weighings sufficing and three being
necessary is harder.  The remainder of the paper is dedicated to
proving that three is not a tight upper bound; namely that the
Baron can always demonstrate the weight of at least one coin among any
$n$ in at most two weighings.  Section~\ref{sec:twostheorem} serves
as a signpost by restating the theorem, and Section~\ref{sec:notation}
briefly introduces some notation we will use subsequently.  The actual
proof is sufficiently involved that we break it down into a separate
Section~\ref{sec:preliminaries} for preliminaries, and then the
proof itself in Section~\ref{sec:main-proof}.

Finally, we close with Section~\ref{sec:discussion} for some
generalizations and ideas for future research.

\section{Baron M\"{u}nchhausen's Sequence}\label{sec:sequence}

Baron M\"{u}nchhausen's sequence $a(n)$ is defined as follows:

\begin{quote}
Let $n$ coins weighing $1, 2,\ldots, n$ grams be given.  Suppose Baron
M\"{u}nchhausen knows which coin wieghs how much, but his audience
does not.  Then $a(n)$ is the minimum number of weighings the Baron
must conduct on a balance scale, so as to unequivocally demonstrate
the weight of at least one of the coins.
\end{quote}
The original Olympiad puzzle is asking whether $a(8) = 1$.

\subsection{Examples}

Let us see what happens for small indices. 

If $n=1$ the Baron does not need to prove anything, as there is just
one coin weighing 1 gram.

For $n=2$ one weighing is enough.  Place one coin on the left cup of
the scale and one on the right, after which everybody knows that the
lighter coin weighs 1 gram and the heavier coin weighs 2 grams.

For $n=3$, the Baron can put the 1-gram and 2-gram coins on one cup
and the 3-gram coin on the other cup.  The cups will balance.  The
only way for the cups to balance is for the lone coin to
weigh 3 grams.

For $n=4$, one weighing is enough.  The Baron puts the 1-gram and
2-gram coins on one cup and the 4-gram coin on the other cup. The only
way for one coin out of four to be strictly heavier than two others
from the set is for it to be the 4-gram coin.  The 3-gram is also
uniquely identified by the method of elimination.

For $n=5$, Baron M\"{u}nchhausen cannot do it in one weighing.  This example is small
enough to prove exhaustively: every possible outcome of every possible
weighing admits of multiple assignments of weights to coins, which do
not fix the weight of any one coin.  For an example of the reasoning,
suppose the Baron puts 1 coin on one cup and 2 coins on another, and shows that
the single coin is heavier.  All of $1+2<4$, $2+1<5$, and $1+3<5$ are
consistent with this data, so no coin is uniquely identified.
Checking all the other cases, as usual, is left to the reader.

For $n=6$, one weighing is enough. This case is similar to the case of
$n=3$: $1+2+3 = 6$.

For $n=7$, one weighing is enough. This case is similar to the case of
$n=4$: $1+2+3 < 7$.

For $n=8$,\label{puzzle-solution} the original Olympiad problem asks whether one weighing is
enough.  It is---Baron M\"{u}nchhausen can convince his audience by
placing all the coins weighing 1 through 5 grams on one cup of the
scale, and the two coins weighing 7 and 8 grams on the other.  The
only way for two coins with weights from the set $1, 2,\ldots, 8$ to
balance five is for the two to weigh the most they can, at $7 + 8 =
15$ grams, and for the five to weigh the least they can at $1 + 2 + 3
+ 4 + 5 = 15$ grams.  This arrangement leaves excactly one coin off
the scale, whose weight, by elimination, must be 6 grams.

So Baron M\"{u}nchhausen's sequence begins with 0, 1, 1, 1, 2, 1, 1, 1.  It also turns
out that these examples illustrate all the methods of proving the
weight of one coin with just one weighing; we will prove this in
Section~\ref{sec:ones}.

\section{Three Weighings are Always Enough}\label{sec:three}

In most of the coin problems we remember from childhood \cite{Guy} the number
of weighings needed to solve the problem grows logarithmically
with the number of coins. Thus, our upper bound theorem may
come as a surprise:

\begin{theorem}\label{maximum}
$a(n) \le 3$.
\end{theorem}

Before going into the proof we would like to introduce a little
notation. First, we denote the $x$-th triangular number $x(x+1)/2$ by
$T_x$.

We already did this in our example weighings, but we would like to
make official the fact that, when describing weighings the Baron
should carry out, we will denote a coin weighing $i$ grams with just
the number $i$.  In addition, we will use round brackets to denote one
coin of the weight indicated by the expression enclosed in the
brackets.  We need this notation to distinguish $i+1$, which represents
two coins of weight $i$ and 1 on some cup, from $(i+1)$, which represents
one coin of weight $i+1$ on some cup.

\begin{proof}
We know, since Carl Friedrich Gauss proved it in 1796, \cite{Gauss, Dunnington} that any
number $n$ can be represented as a sum of not more than 3 triangular
numbers. Let $n = T_i + T_j +T_k$, where $T_i \le T_j \le T_k$ are
triangular numbers with indices $i$, $j$, and $k$.

Barring special cases, Baron M\"{u}nchhausen can display a sequence of
three weighings with one coin on the right cup each: first
\[ 1+2+3+\dots+k=T_k, \]
then
\[ T_k + 1+2+3+\dots+j=(T_k+T_j), \]
and finally
\[ (T_k + T_j) + 1+2+3+\dots+i=n. \]
The first weighing demonstrates a lower bound of $T_k$ on the weight of
the coin the Baron put on the right.  Since he then reuses that coin on the
left, the second weighing demonstrates a lower bound of $T_k + T_j$ on
the weight of the coin that goes on the right in the second weighing.
Since he then reuses \emph{that} coin in the third weighing, the audience finds a
lower bound of $T_k + T_j + T_i$ on the coin on the right hand side of
the last weighing.  But since there are only $n$ coins, there is
already an upper bound of $n$ on that coin's weight, so the assumed
equality $T_k + T_j + T_i = n$ determines that coin's weight
completely (as well as the weights of the $T_k$- and $(T_k + T_j)$-gram
coins).

The Baron should start with the largest triangular number to make sure
that he will not need any coin to appear in the same weighing twice: since $k$
is the largest index, the coin $T_k$ will not be in the sequence $1,
2,\dots, j$, and the coin $(T_k+T_j)$ will not be in the sequence $1,
2,\dots, i$.

When does this procedure fail?  If $i=0$, the last weighing is
impossible but also redundant, because it would be asking to weigh the
$n$-gram coin against itself.  If $j=0$, the second weighing is likewise
impossible and redundant.  If $k=0$, there is nothing to prove because
$n=0$ as well.  Finally, if $k=1$, the first weighing is impossible
because $T_k = 1$ occurs in $1\ldots k$, but in this case $n$ is at
most $3$, and that can be solved in fewer than three weighings anyway.
\end{proof}

\section{When does One Weighing Suffice?}\label{sec:ones}

Let us look at where ones occur in this sequence.  We characterize the
possible weighings that could determine a coin by themselves.  The
determination of the numbers $n$ of coins that admit such weighings,
i.e., for which $a(n)=1$, is a direct consequence.

\subsection{One-Weighing Determinations}

\begin{theorem}\label{ones}
The weight of some coin can be confirmed with just one weighing if and
only if all of:
\begin{enumerate}
\item one cup of that weighing contains all the coins with weights
  from 1 to some i;
\item the other cup contains all the coins with weights from some j to
  n;
\item Either the scale balances, or the cup containing the 1-gram coin
  is lighter by one gram; and
\item At least one cup contains exactly one coin, or exactly one coin
  is left off the scale.
\end{enumerate}
\end{theorem}

Why can such a weighing be convincing?  In general, $i$ will be much
larger than $n-j$.  The only way for so few coins to weigh as much as
(or more than) so many will be for the few to be the heaviest and the
many to be the lightest.  We show in the proof that those are exactly
the convincing weighing structures; thereafter, in Section~\ref{sec:one-indices},
we discuss the
circumstances under which such a weighing exists and can therefore
determine the weight of a single coin.

\begin{proof}
What does it mean for Baron M\"{u}nchhausen to convince his audience of the
weight $k$ of some coin, using just one weighing? From the perspective of the
audience, a weighing is a number of coins in one cup, a number of coins
in the other cup, and a number of coins not on the scale, together
with the result the scale shows (one or the other cup heavier, or both
the same weight). For the audience to be convinced of the weight of some
particular coin, it must therefore be the case that all possible
arrangements of coin weights consistent with that data agree on the
weight $k$ of the coin in question.

\textbf{``If'' direction.}  Suppose all the conditions in the theorem
statement are met.  Then one cup of the scale contains $i$ coins, and
the other $n-j+1$ coins.  Suppose, for definiteness, that these are
the left and right sides of the scale, respectively.  The least that
$i$ coins can weigh is $T_i$.  The most that $n-j+1$ coins can weigh
is $T_n - T_{j-1}$.  The audience can compute these numbers; and if
they are equal, then the audience can conclude that, for the scale to
balance, the left cup must have exactly the coins of weights $1\ldots
i$, and the right cup must have exactly the coins of weights $j\ldots
n$.  Likewise, if $T_n - T_{j-1}$ exceeds $T_i$ by one, the same
allocation is the only way for the scale to indicate that the cup with
$i$ coins is lighter.

If either of the sets $1\ldots i$ or $j\ldots n$ is a singleton, that
determines the weight of that coin directly; otherwise, condition
{\it(iv)} requires there to be only one coin off the scale, and the
weight of that one remaining coin can be determined by the process of
elimination.

\textbf{``Only if'' direction.}  Our proof strategy is to
look for ways to alter a given arrangement of coin weights so as to
change the weight given to the coin whose weight is being
demonstrated; the requirement that all such alterations are impossible
yields the desired constraints on convincing weighings.

First, obviously, the coin whose weight $k$ the Baron is trying to
confirm has to be alone in its group: either alone on some cup or the
only coin not on the scale. After that observation we divide the proof of the
theorem into several cases.

\textbf{Case 1.} The $k$-gram coin is on a cup and the scale is balanced.
Then by above $k$ is alone on its cup.  We want to show two things:
$k = n$, and the coins on the other cup weigh 1, 2, ..., $i$ grams for
some $i$. For the first part, observe that if $k < n$, then the coin
with weight $k+1$ must not be on the scale (otherwise it would
overbalance coin $k$). Therefore, we can substitute coin $k+1$ for
coin $k$, and substitute a coin one gram heavier for the heaviest coin
that was on the other cup, and produce thereby a different weight arrangement
with the same observable characteristics but a different weight for
the coin the Baron claims has weight $k$.

To prove the second part, suppose the contrary. Then it is possible to
substitute a coin 1 gram lighter for one of the coins on the other
cup. Now, if coin $k-1$ is not on the scale, we can also substitute
$k-1$ for $k$, and again produce a different arrangement with the same
observable characteristics but a different weight for the coin labeled
$k$. On the other hand, if $k-1$ is on the scale but $k-2$ is not,
then we can substitute $k-2$ for $k-1$ and then $k-1$ for $k$ and the
weighing is again unconvincing. Finally, if both $k-1$ and $k-2$ are
on the scale, and yet they balance $k$, then $k=3$ and the theorem
holds.

Consequently, $k = n = 1 + 2 + \dots + i = T_i$ is a triangular number.

\textbf{Case 2.} The $k$-gram coin is on the lighter cup of the
scale. Then: first, $k = 1$, because otherwise we could swap $k$ and
the 1-gram coin, making the light cup lighter and the heavy cup
heavier or unaffected; second, the 2-gram coin is on the heavy cup and
is the only coin on the heavy cup, because otherwise we could swap $k$
with the 2-gram coin and not change the weights by enough to affect
the imbalance; and finally $n = 2$ because otherwise we could change
the weighing $1 < 2$ into $2 < 3$.

Thus the theorem holds, and the only example of this case is $k = 1, n
= 2$.

\textbf{Case 3.} The $k$-gram coin is on the heavier cup of the scale. Then
$k = n$ and the lighter cup consists of some collection of the
lightest available coins, by the same argument as Case 1 (but even
easier, because there is no need to maintain the
balance). Furthermore, $k$ must weigh exactly 1 gram more than the
lighter cup, because otherwise, $k-1$ is not on the lighter cup and
can be substituted for $k$, making the weighing unconvincing.

Consequently, $k = n = (1 + 2 + \dots + i) + 1 = T_i + 1$ is one more than a
triangular number.

\textbf{Case 4.} The $k$-gram coin is not on a cup and the scale is not
balanced. Then, since $k$ must be off the scale by itself, all the
other coins must be on one cup or the other. Furthermore, all coins
heavier than $k$ must be on the heavier cup, because otherwise we
could make the lighter cup even lighter by substituting $k$ for one of
those coins. Likewise, all coins lighter than $k$ must be on the
lighter cup, because otherwise we could make the heavier cup even
heavier by substituting $k$ for one of those coins. So the theorem
holds; and furthermore, the cups must again differ in weight by
exactly 1 gram, because otherwise we could swap $k$ with either $k-1$
or $k+1$ without changing the weights enough to affect the result on
the scale.

Consequently, the weight of the lighter cup is $k(k-1)/2$, and the
weight of the heavier cup is $k(k-1)/2 + 1$. Thus the total weight of
all the coins is $n(n+1)/2 = k(k-1)/2 + k + (k(k-1)/2 + 1) =
k^2+1$. In other words, case 4 is possible iff $n$ is the index of a
triangular number that is one greater than a square.

\textbf{Case 5.} The $k$-gram coin is not on a cup and the scale is
balanced. This case is hairier than all the others combined, so we
will take it slowly (noting first that all the coins besides $k$ must be
on some cup).

\begin{lemma}\label{lemma1}
The two coins $k-1$ and $k-2$ must be on the same cup, if they exist
(that is, if $k > 2$). Likewise $k-2$ and $k-4$; $k+1$ and $k+2$; and
$k+2$ and $k+4$.
\end{lemma}

\begin{proof}
Suppose the two coins $k-1$ and $k-2$ are not on the same cup. Then we can rotate $k$, $k-1$, and $k-2$, that
is, put $k$ on the cup with $k-1$, put $k-1$ on the cup with $k-2$,
and take $k-2$ off the scale. This makes both cups heavier by one
gram, producing a weighing with the same outward characteristics as
the one we started with, but a different coin off the scale. The same
argument applies to the other three pairs of coins we are interested
in, mutatis mutandis.
\end{proof}

\begin{lemma}\label{lemma2}
The four coins $k-1$, $k-2$, $k-3$ and $k-4$ must be on the same cup
if they exist (that is, if $k \ge 5$).
\end{lemma}

\begin{proof}
By Lemma \ref{lemma1}, the three coins $k-1$, $k-2$, and $k-4$ must be
on the same cup. Suppose coin $k-3$ is on the other cup. Then we can
swap $k-1$ with $k-3$ and $k$ with $k-4$. Each cup becomes lighter by 2
grams without changing the number of coins present, the balance is
maintained, and the Baron's audience is not convinced.
\end{proof}

\begin{lemma}\label{lemma3}
If coin $k-4$ exists, that is if $k \ge 5$, all coins lighter than k
must be on the same cup.
\end{lemma}

\begin{proof}
By Lemma \ref{lemma2}, the four coins $k-1$, $k-2$, $k-3$ and $k-4$
must be on the same cup. Suppose some lighter coin is on the other
cup. Call the heaviest such coin $c$. Then, by choice of $c$, the coin
with weight $c+1$ is on the same cup as the cluster $k-1\dots k-4$,
and is distinct from coin $k-2$. We can therefore swap $c$ with $c+1$ and swap $k$ with
$k-2$. This increases the weight on both cups by 1 gram without
changing how many coins are on each, but moves $k$ onto the scale. The
Baron's audience is again unconvinced.
\end{proof}

\begin{lemma}\label{lemma4}
Theorem \ref{ones} is true for $k \ge 5$.
\end{lemma}

\begin{proof}
By Lemma \ref{lemma3}, all coins lighter than $k$ must be on the same
cup. Further, if a coin with weight $k+4$ exists, then by the
symmetric version of Lemma \ref{lemma3}, all coins heavier than $k$
must also be on the same cup. They must be on the other cup from the
coins lighter than $k$ because otherwise the scale would not balance, and
the theorem is true.

If no coin with weight $k+4$ exists, that is, if $n \le k+3$, how can
the theorem be false? All the coins lighter than $k$ must be on one
cup, and their total weight is $k(k-1)/2$. Further, in order to
falsify the theorem, at least one of the coins heavier than $k$ must
also be on that same cup. So the minimum weight of that cup is now
$k(k-1)/2 + k+1$. But we only have at most two coins for the other
cup, whose total weight is at most $k+2 + k+3 = 2k + 5$. For the scale
to even have a chance of balancing, we must have
$$k(k-1)/2 + k+1 \le 2k + 5 \Leftrightarrow k^2 - 3k - 8 \le 0.$$

Finding the largest root of that quadratic we see that $k < 5$.

So for $k \ge 5$, the collection of all coins lighter than $k$ is
heavy enough that either one needs all the coins heavier than $k$ to
balance them, or there are enough coins heavier than $k$ that the
theorem is true by symmetric application of Lemma \ref{lemma3}.

Completion of Case 5. It remains to check the case for $k < 5$. If $n
> k+3$, then coin $k+4$ exists. If so, all the coins heaver than $k$
must be on the same cup. Furthermore, since $k$ is so small, they will
together weigh more than half the available weight, so the scale will
be unbalanceable. So $k < 5$ and $n \le k+3 \le 7$.

For lack of any better creativity, we will tackle the remaining
portion of the problem by complete enumeration of the possible cases,
except for the one observation that, to balance the scale with just
the coin $k$ off it, the total weight of the remaining coins, 
$n(n+1)/2 - k$, must be even. This observation cuts our remaining work
in half. Now to it.

\textbf{Case 5; Seven Coins: $n = 7$.} Then $5 > k \ge n - 3 = 4$, so
$k = 4$. Then the weight on each cup must be 12. One of the cups must
contain the 7 coin, and no cup can contain the 4 coin, so the only two
weighings the Baron could try are $7 + 5 = 1 + 2 + 3 + 6$, and $7 + 3
+ 2 = 1 + 5 + 6$. But the first of those is unconvincing because $k+1
= 5$ is not on the same cup as $k+2 = 6$, and the second because it
has the same shape as $7 + 3 + 1 = 2 + 4 + 5$ (leaving out the 6-gram
coin instead of the asserted 4-gram coin).

\textbf{Case 5; Six Coins: $n = 6$.} Then $5 > k \ge n - 3 = 3$, and
$n(n+1)/2 = 21$ is odd, so $k$ must also be odd. Therefore $k=3$, and
the weight on each cup must be 9. The 6-gram coin has to be on a cup
and the 3-gram coin is by presumption out, so the Baron's only chance
is the weighing $6 + 2 + 1 = 4 + 5$, but that does not convince his
skeptical audience because it looks too much like the weighing $1 + 3 +
4 = 6 + 2$.

\textbf{Case 5; Five Coins: $n = 5$.} Then $5 > k \ge n - 3 = 2$, and
$n(n+1)/2 = 15$ is odd, so $k$ must also be odd. Therefore $k=3$, and
the weight on each cup must be 6. The only way to do that is the
weighing $5 + 1 = 2 + 4$, which does not convince the Baron's audience
because it looks too much like $1 + 4 = 2 + 3$.

\textbf{Case 5; Four Coins: $n = 4$.} Then the only way to balance a
scale using all but one coin is to put two coins on one cup and one on
the other. The only two such weighings that balance are $1 + 2 = 3$
and $1 + 3 = 4$, but they leave different coins off the scale.

The remaining cases, $n < 4$, are even easier. That concludes the
proof of Case 5.

Consequently, by an argument similar to the one in case 4 we can show
that any number $n$ of coins to which case 5 applies must be the
index of a square triangular number.
\end{proof}

This concludes the proof of Theorem~\ref{ones}.
\end{proof}

\subsection{The Indices of Ones}
\label{sec:one-indices}

While proving the theorem we accumulated descriptions of all possible
numbers of coins that allow the Baron to confirm a coin in one
weighing.  We collect that list here to finish the description of the
indices of ones in Baron M\"{u}nchhausen's sequence $a(n)$.  The
following list corresponds to the five cases in the proof of Theorem
\ref{ones}:

\begin{enumerate}
\item $n$ is a triangular number: $n = T_i$. Then the weighing $1+2+3
  + \ldots + i = n$ proves weight of the $n$-gram coin. For example,
  for six coins the weighing is $1+2+3 = 6$.
\item $n = 2$. The weighing $1 < 2$ proves the weight of both coins.
\item $n$ is a triangular number plus one: $n=T_i+1$. Then the
  weighing $1+2+3 + \ldots + i < n$ proves the weight of the $n$-gram
  coin. For example, for seven coins the weighing is $1+2+3 < 7$.
\item $n$ is the index of a triangular number that is a square plus
  one: $T_n=k^2+1$. Then the weighing $1+2+3 + \ldots + (k-1) < (k+1)
  + \ldots + n$ proves the weight of the $k$-gram coin. For example, the fourth
  triangular number, which is equal to ten, is one greater than a
  square. Hence the weighing $1+2 < 4$ can identify the coin that is
  not on the cup. The next number like this is 25, and the
  corresponding weighing is $1+2+\dots+17 < 19+20 +\dots+25$.
\item $n$ is the index of a square triangular number: $T_n=k^2$. Then
  the weighing $1+2+3 + \ldots + (k-1) = (k+1) + \ldots + n$ proves
  the weight of the $k$-gram coin. For example, we know that the 8th
  triangular number is 36, which is a square: our original problem
  corresponds to this case.
\end{enumerate}

The sequence of indices of ones in the sequence $a(n)$ starts as: 1,
2, 3, 4, 6, 7, 8, 10, 11, 15, 16, 21, 22, 25, 28, 29, 36, 37, 45, 46,
49, 55, 56, 66, 67, 78, 79, 91, 92.

\subsection{Discussion}

If we have four coins, then the same weighing $1+2 < 4$ identifies two
coins: the coin that weighs three grams and is not on the scale and
the coin weighing four grams that is in a cup. The other case like
this is for $n=2$. Comparing the two coins to each other we can
identify both of them. It is clear that there are no other cases like
this. Indeed, for the same weighing to identify two different coins,
it must be the $n$-gram coin on a cup, and the $(n-1)$-gram coin off
the scale. From here we can see that $n$ cannot be very big.

As usual, we want to give our readers something to think about. We
have given you the list of four sequences that correspond to
four cases describing all the numbers for
which the Baron can prove the weight of one coin in one weighing. Does
there exist a number greater than four that belongs to two of these
sequences? In other words, does there exist a total number of coins
such that the Baron can have two different one-weighing proofs for two
different coins?

\section{Two Weighings are Always Enough}\label{sec:twostheorem}

Our main theorem states that Baron M\"{u}nchhausen never needs three
weighings, for two are always enough.

\begin{theorem}\label{twos}
$a(n) \leq 2$.
\end{theorem}

\section{Notation}\label{sec:notation}

In the proofs of the previous theorems and lemmas, we have already seen
some recurring elements: triangular numbers are important;
contiguous ranges of coins are important.  Additional common elements
will arise in our further explorations, so we introduce some
notation now to more easily manipulate them.

As before the coins are numbered according to their weights, and we
will continue to use the number $i$ to denote an $i$-gram coin on a cup,
using round brackets as before to distinguish a single coin of a weight
we need to compute from two separate coins.  For example, (1+2) means
the 3-gram coin, whereas 1+2 means the 1-gram and the 2-gram coin.

We will also continue to use the notation $T_x$ to denote the $x$th
triangular number $x(x+1)/2$.

We introduce the notation $[x\dots y]$ for the set of all consecutive
coins between $x$ and $y$, inclusive; and we will occasionally
construct weighings with set notation.  Inside expressions in square
brackets we will not parenthesize computations: $[3+4\ldots 11-1]$ is
the set of coins weighing from $7$ to $10$, inclusive, and does not
include the coins $3$, $4$, $11$, or $1$.

If $A$ denotes a set of coins, then $|A|$ denotes the total weight of
those coins (not the cardinality of the set).

When representing a weighing as an equality/inequality we will refer
to the left and right sides of the equality/inequality as the left and
right cups of the weighing, respectively.

\section{Preliminaries}\label{sec:preliminaries}

Before we proceed with the main section of the proof, we will prove
two lemmas that we are going to need, and that will demonstrate the
machinery we will use to prove the main Theorem \ref{twos}.

\begin{lemma} \label{lem:sum-two-triangles}
If $n$, $n-1$, or $n-2$ is a sum of two triangular numbers, then the
Baron can demonstrate the weight of the $n$-gram coin in two
weighings.
\end{lemma}

\begin{proof}
This is a direct corollary of the argument used to prove
Theorem~\ref{maximum}.  If $n = T_a + T_b$, that argument applies
exactly.  In the other two cases, the Baron can make judicious use of
unbalanced weighings.

If $n = T_a + T_b + 1$, for $a \leq b$, then one of the weighings
needs to be unbalanced, for example
\beas
[1\ldots b] & = & T_b \\ \mbox{}
[1\ldots a] + T_b & < & n.
\eeas
If $n = T_a + T_b + 2$, then both weighings should be unbalanced:
\beas
[1\ldots b] & < & (T_b+1) \\ \mbox{}
[1\ldots a] + (T_b+1) & < & n.
\eeas
\end{proof}

Since triangular numbers are pretty dense among the small integers,
Theorem~\ref{ones} and Lemma~\ref{lem:sum-two-triangles} account for
many small $n$.  This is good, because the main proof in
Section~\ref{sec:main-proof} does not go through for small $n$.  In
particular, the reader is invited to verify that the smallest $n$ that
does not fall under the purview of Lemma~\ref{lem:sum-two-triangles}
is $n = 54$; for example by consulting sequence A020756 in OEIS \cite{OEIS}.

The following covers a special case we will encounter in the main
proof, and coincidentally demonstrates the argument we will use in the
main proof that the complicated weighings we will present will, in
fact, convince the Baron's audience.

\begin{lemma}
\label{lem:two-close-triangles}
If there exists an $a$ such that $2n = T_a + T_{a+1}$, the Baron can
prove the weight of the $n$-gram coin in two weighings.
\end{lemma}
\begin{proof}
We know that $T_a < n < T_{a+1}$ and, in fact, $n = T_a + \frac{a+1}{2}$.
Suppose we can find coins $x$ and $y$ with $a+1 < x < y = x + \frac{a+1}{2} < n$.  Then
the Baron can present the following two weighings:
\[ [1\ldots a] + y = x + n \]
and
\[ [1\ldots a+1] + x = y + n. \]
They will balance by the choice of $x$ and $y$.  Why will they
convince the Baron's audience?

Let the audience consider the sum of the two weighings.  The coins $x$
and $y$ appear on both sides of the sum, so they do not affect the
balance of the total.  Besides them, $a$ coins appeared twice on the
left, and one additional coin appeared once on the left; and this huge
pile of stuff was balanced by just two appearances of a single coin on
the right.  How is this possible?  The least possible total weight of
the left-hand sides (except $x$ and $y$) occurs if the coins that
appeared twice have weights $[1\ldots a]$, and the coin that appeared
once has weight $a+1$, for a total weight of $T_a + T_{a+1}$.  The
greatest possible total weight of the right-hand sides (again
excluding $x$ and $y$) occurs if the solitary coin on the right weighs
$n$ grams.  But the known fact that $2n = T_a + T_{a+1}$ guarantees
that, even in this extreme case the scale will just barely balance; so
any other set of weights would cause the left cup to overbalance the
right in at least one of the weighings Baron M\"{u}nchhausen conducts.  Therefore,
since, in fact, neither left cup overbalanced its corresponding right
cup, the Baron's audience is forced to conclude that the solitary coin
on the right must weigh $n$ grams, as was the Baron's intention.
(Coincidentally, this scenario also proves the weight of the $a+1$
coin.)

Now, when can we find such coins $x$ and $y$?  We can safely take the
$a+2$ coin for $x$.  Then the desired $y$ coin will exist if $n > a +
2 + \frac{a+1}{2}$, which is equivalent to $T_a > a+2$, which holds for $a \geq 3$. The last condition translates into $n \geq 8$.

Smaller $n$ are covered by Lemma \ref{lem:sum-two-triangles}.
\end{proof}

\section{Proof of the Main Theorem}
\label{sec:main-proof}

There are two magical steps.  First, let $a \leq b \leq c$ be such
that
\be T_a + T_b + T_c = n + T_n. \label{eqn:three-triangles}
\ee
By the triangular number theorem, proved by Gauss in his diary, \cite{Gauss, Dunnington}
such a decomposition of $T_n + n$
into three triangular numbers is always possible.  We should remark
at this point that $c > n$ would imply $T_c \geq T_{n+1} > T_n + n$ so
is impossible; and that $c = n$ would imply $T_c = T_n$ so $T_a + T_b = n$,
allowing the Baron to proceed by the method in Lemma~\ref{lem:sum-two-triangles}.
So we can assume $c < n$.

Second, let us try to represent $T_c - n$ as the sum of some subset
$S$ of weights from the range $[a+1 \ldots n-1]$.  Now there are three
non-magical steps.  We will prove that if such a representation exists,
then the Baron can convince his audience of the weight of the $n$-gram
coin in two weighings, by a particular method to be described
forthwith; then we will take some time to study the properties of sums
of subsets of ranges of integers; and then at the last we will
systematically examine possible choices of $a$, $b$, and $c$, and
prove that the above-mentioned subset $S$ really does exist, except in one case,
for which Lemma~\ref{lem:two-close-triangles}
supplies an alternate method of solution.

\subsection{Step 1: What to do with $S$}

\begin{lemma}
Let $a$, $b$, and $c$ satisfy
\[ T_a + T_b + T_c = n + T_n. \]
Let $S$ be a subset of $[a+1 \ldots n-1]$ for which
\[ |S| = T_c - n. \]
Then there exist two weighings that uniquely identify the $n$-gram coin.
\end{lemma}
\begin{proof}
Let $\bar S$ denote the complement of $S$ in $[a+1 \ldots n-1]$.  We want to make a useful weighing
out of the assumption about $S$, so let us proceed as follows:
\[ T_c = |S| + n, \]
which we rewrite in terms of weights of coins
\[ [1 \ldots a] + [a+1 \ldots c] = S \cap [a+1\ldots c] + S \cap [c+1\ldots n-1] +n, \]
and cancel coins appearing on both sides to get
\be
[1 \ldots a] + \bar S \cap [a+1\ldots c] = S \cap [c+1\ldots n-1] + n. \label{eqn:shuffle}
\ee
Observe that (\ref{eqn:shuffle}) now forms a legal
weighing; and indeed, let us take it to be the first weighing.

Now we want to make another weighing that will, together with
(\ref{eqn:shuffle}), demonstrate the weight of the $n$-gram coin.
Let us begin by massaging (\ref{eqn:three-triangles}):
\beas
T_a + T_b + T_c & = & n + T_n \\
T_a + T_b & = & T_{n-1} - T_c + 2n \\
T_a + T_b + |\bar S \cap [b+1\ldots c]| & = & |\bar S \cap [b+1\ldots c]| + |[c+1\ldots n-1]| + 2n.
\eeas
Now, converting the last equality into coins and subtracting (\ref{eqn:shuffle}), we get
\be
[1\ldots a] + S\cap[a+1\ldots b] = \bar S\cap[b+1\ldots c] + \bar S\cap [c+1\ldots n-1] + n. \label{eqn:weigh2}
\ee
Again, each coin occurs at most once, so the Baron can legitimately
take (\ref{eqn:weigh2}) as his second weighing.

We have just shown that (\ref{eqn:shuffle}) and (\ref{eqn:weigh2})
represent four sets of coins that can be weighed against each other in
the indicated pattern, and that the scale will balance if they are.
Now, why do these two weighings uniquely identify the $n$-gram coin?
Consider which coins appear on which sides of those two equations.
Let $L_1$ and $R_1$ be the left- and right-hand sides of the first weighing (\ref{eqn:shuffle}),
respectively, and likewise $L_2$ and $R_2$ for the second weighing (\ref{eqn:weigh2}).
Also, let $O_1$ and $O_2$ be the sets of coins that do not participate
in (\ref{eqn:shuffle}) and (\ref{eqn:weigh2}), respectively.
Then
\beas
              [1\ldots a]   & = & L_1 \cap L_2, \\
\bar S \cap [a+1\ldots b]   & = & L_1 \cap O_2, \\
     S \cap [a+1\ldots b]   & = & O_1 \cap L_2, \\
     S \cap [b+1\ldots c]   & = & O_1 \cap O_2, \\
\bar S \cap [b+1\ldots c]   & = & L_1 \cap R_2, \\
\bar S \cap [c+1\ldots n-1] & = & O_1 \cap R_2, \\
     S \cap [c+1\ldots n-1] & = & R_1 \cap O_2, \\
                       n    & = & R_1 \cap R_2.
\eeas

Seeing the two weighings (\ref{eqn:shuffle}) and (\ref{eqn:weigh2}), Baron M\"{u}nchhausen's audience reasons analagously
to how they did in the proof of Lemma~\ref{lem:two-close-triangles}.
They consider the sum of the two weighings, which tells them
\[ |L_1| + |L_2| = |R_1| + |R_2|. \]
They see that some coins, namely $L_1 \cap R_2$, (which the Baron
knows to be $\bar S \cap [b+1\ldots c]$) appeared first on the left
and then on the right, so those coins do not affect the balance of
the sum.  The audience also sees that 
\begin{enumerate}
\item $a$ coins appeared on the left both times ($L_1 \cap L_2$);
\item $b-a$ coins appeared on the left once and never on the right
  ($(L_1 \cap O_2) \cup (O_1 \cap L_2)$);
\item $n-1-c$ coins appeared on the right once and never on the left
  ($(R_1 \cap O_2) \cup (O_1 \cap R_2)$); and
\item just one coin appeared on the right both times ($R_1 \cap R_2$).
\end{enumerate}
Now, $a$ and $b-a$ are going to be much bigger than $n-c-1$ and $1$,
so the audience will be surprised that so many coins can be balanced by
so few.  And they will wonder how to minimize the total weight
\[ 2|L_1 \cap L_2| + |(L_1 \cap O_2) \cup (O_1 \cap L_2)| \]
of the many, and how to maximize the total weight
\[ |(R_1 \cap O_2) \cup (O_1 \cap R_2)| + 2|R_1 \cap R_2| \]
of the few.  And they will see that to do this, they must
\begin{enumerate}
\item let the coins in $L_1 \cap L_2$ have the weights $[1\ldots a]$,
  as they occur on the left twice;
\item let the coins in $(L_1 \cap O_2) \cup (O_1 \cap L_2)$ have the
  weights $[a+1\ldots b]$, as they occur on the left once;
\item let the coins in $(R_1 \cap O_2) \cup (O_1 \cap R_2)$ have the
  weights $[c+1\ldots n-1]$, as they occur on the right once; and
\item let the sole coin in $R_1 \cap R_2$ have weight $n$, as it occurs
  on the right twice.
\end{enumerate}
And then they will see from (\ref{eqn:three-triangles}), which
can be rewritten as
\beas
T_a + T_b + T_c & = & T_n + n \\
T_a + T_b & = & T_{n-1} - T_c + 2n \\
2|[1\ldots a]| + |[a+1\ldots b]| & = & |[c+1\ldots n-1]| + 2n
\eeas
that even if they minimize the left and maximize the right, the scale
will just barely balance.  And then they will know that any other
weights than those would have made the left heavier than the right,
and since the scale did balance, those are the weights that must have
been, and they will wonder in awe at the Baron's skill in convincing
them of the weight of his chosen coin out of $n$ in only two
weighings.
\end{proof}

We have established that the existence of a subset $S$ of $[a+1\ldots
  n-1]$ that adds up to $|S| = T_c - n$ suffices to let the Baron
convince his audience of the weight of the coin labeled $n$ in two
weighings.  Now, when does such a subset reliably exist?

\subsection{Step 2: Sums of subsets of ranges}

To answer this question, let us study the behavior of sums of subsets
of ranges of positive integers in general.  The results of this
segment probably generalize to negative integers and beyond, and are
probably published in many books, but we decided that it is faster to
derive them ourselves than drive to the library.

Suppose we have
some range of integers $[s\ldots t]$.  What are the possible sums of
its subsets?  First, what are the possible sums of subsets of a fixed
size, say $k$?  Well, the smallest sum of $k$ elements of $[s\ldots t]$
is of course the sum of the $k$ smallest elements of $[s\ldots t]$:
\[ s + (s+1) + \ldots + (s+(k-1)) = ks + T_{k-1} = T_{s+k-1} - T_{s-1}. \]
The largest sum of $k$ elements of $[s\ldots t]$ is of course
the sum of the $k$ largest elements of $[s\ldots t]$:
\[ t + (t-1) + \ldots + (t-(k-1)) = kt - T_{k-1} = T_t - T_{t-k}. \]
What is more, given any subset $K$ of $k$ elements of $[s\ldots t]$
that are not the $k$ largest, we can change one of them for an element
one larger that was not in $K$, thus producing a subset whose sum is larger
by one.  Since we can walk all the way from the $k$ smallest elements to
the $k$ largest elements by increments of one, the possible sums cover the
whole range between the least and the greatest possible values, and we
have just proven
\begin{lemma}
The set of possible sums of subsets of size $k$ of a range $[s\ldots
  t]$ is exactly the range $[ks + T_{k-1} \ldots kt - T_{k-1}]$.
\end{lemma}

Now, what about the overall behavior of subsets of any size?  Well,
subsets of size $k$ form a contiguous range, and subsets of size $k+1$
also form a contiguous range.  Do those ranges join to form a larger
range, or is there a gap?  In other words, is one plus the maximum sum
of subsets of size $k$ a possible sum of subsets of size $k+1$?  This
will be true if and only if replacing the $k$ largest elements of
$[s\ldots t]$ with the $k+1$ smallest does not increase the sum by
more than 1, or
\be kt - T_{k-1} + 1 \geq (k+1)s + T_k. \label{eq:rangeoverlap} \ee

Moreover, if turning the $k$ largest elements into the $k+1$ smallest
elements does not cause an increase exceeding 1, the same will hold
for $k+1$, as long as $k$ is less than the middle point of the segment
$[s\ldots t]$: $k < \frac{t-s}{2}$. Indeed if we sum up the inequality
(\ref{eq:rangeoverlap}) with $t-k \geq s + k+1$, we get
\[ (k+1)t - T_{k} + 1 \geq (k+2)s + T_{k+1}, \]
the condition that the next two ranges overlap.

This means that if the
$k$-subset range overlaps the $k+1$-subset range, then all larger
ranges will also overlap, at least until size $\frac{t-s}{2}$.  Also,
subsets of size greater than $\frac{t-s}{2}$ overlap symmetrically to
their smaller counterparts, because the sum of any such subset is just
the total sum of all numbers between $s$ and $t$ minus the sum of the
complement of that subset.  This demonstrates
\begin{lemma}
If $k < (t-s)/2$ is such that 
\[ kt - T_{k-1} + 1 \geq (k+1)s + T_k, \]
then the possible sums of subsets of sizes $[k\ldots t-s-k]$ create a
contiguous range.  In other words, it is possible to find subsets of the
range $[s\ldots t]$ that sum up to any number between $ks +T_{k-1}$
and $(t-s-k)t-T_{t-s-k-1}$.
\end{lemma}

Considering the possibility that ranges may start overlapping from
$k=1$, that is $t - T_0 + 1 \geq 2s + T_1$, leads us to
\begin{corollary}
\label{cor:big-range}
If $t+1 \geq 2s+1$, or, equivalently, $s \leq t/2$, the
subsets of the range $[s\ldots t]$ can achieve any sum in
\[ [s\ldots T_t - T_s]. \]
\end{corollary}
Considering the possibility that ranges may start overlapping from
$k=2$, that is $2t - T_1 + 1 \geq 3s + T_2$, leads us to
\begin{corollary}
\label{cor:small-range}
If $2t \geq 3s+3$, or, equivalently, $s \leq \frac{2}{3}t - 1$, the
subsets of the range $[s\ldots t]$ can achieve any sum in
\[ [2s+1\ldots T_t - T_{s+1}]. \]
\end{corollary}
These two facts will prove invaluable to
characterizing when $T_c-n$, from above, can be achieved as the sum of
some set of coins from a given range.

\subsection{Step 3: Systematic study of possibilities for $a$, $b$, and $c$}

We are now ready to finish this proof.  Recall the setup: The Baron
has $n$ coins; we have made a decomposition into three triangular numbers
\[ n + T_n = T_a + T_b + T_c, \ \ \ a \leq b \leq c; \]
and we know that if we can find a subset $S$ of $[a+1\ldots n-1]$
for which
\[ T_c - n = |S|, \]
the Baron can convince his audience of the weight of the $n$-gram coin in two
weighings.  We also know, from the corollaries above, that
\begin{enumerate}
\item If $2a+3 \leq n$, sums of subsets of $[a+1\ldots n-1]$ include
the range $[a+1\ldots T_{n-1} - T_{a+1}]$, and
\item If $3a+6 \leq 2n-2$, sums of subsets of $[a+1\ldots n-1]$ include
the range $[2a+3\ldots T_{n-1} - T_{a+2}]$.
\end{enumerate}

Here is the main idea for the remainder of this section, and with it,
the proof: As $T_a$ is the smallest number in our decomposition, we
know that $T_a \leq \frac{T_n + n}{3} < \frac{T_{n+1}}{3}$.  We can
conclude from this that $a < \frac{n+1}{\sqrt{3}}$.  Since
$\frac{n+1}{\sqrt{3}}$ grows slower than $\frac{2n}{3}$, for large $n$
we expect the condition $3a+6 \leq 2n-2$ in Corollary
\ref{cor:small-range} to hold.  Then it will suffice to prove that
$T_c - n$ falls into the range $[2a+3\ldots T_{n-1} - T_{a+2}]$.  In
general, the lower bound will be easy; and for the upper bound we will
find that if $T_c$ is large, then $T_a$ will be small, so this
analysis will cover all but a few possibilities for $c$; but these few
will be extreme enough and few enough to handle directly.

Now to it.  Since $a < \frac{n+1}{\sqrt{3}}$,
\[ 3a + 6 < (n+1)\sqrt{3} + 6. \]
For $n \geq 37$,
\[ (n+1)\sqrt{3} + 6 \leq 2n - 2. \]
Combining these two we get the desired
\[ 3a + 6 < 2n - 2, \]
so Corollary~\ref{cor:small-range} applies.  Subsets of $[a+1\ldots n-1]$
can take on all sums in the range $[2a+3\ldots T_{n-1} - T_{a+2}]$.

Does $T_c - n$ fall into this range?  For the upper bound, we have the
sequence of equivalent inequalities, starting with the desired one
\beas
T_c - n & \leq & T_{n-1} - T_{a+2} \\
T_a + T_c + (a+1) + (a+2) & \leq & T_n \\
T_a + T_c + (a+1) + (a+2) + n & \leq & T_n + n = T_a + T_b + T_c \\
n + 2a+3 & \leq & T_b.
\eeas

On the other hand, if $n + 2a+3 \leq T_b$, then
\be T_c - n \geq T_b - n \geq 2a + 3. \ee

So for $n \geq 37$, as long as $T_b \geq n+2a+3$, a subset
$S$ of $[a+1\ldots n-1]$ can be found that sums to $T_c - n$, permitting the
Baron to convince his audience of the weight of the coin labeled $n$.

When can we guarantee that $T_b \geq n+2a+3$? 
We know that $a < \frac{n+1}{\sqrt{3}}$, 
so it is enough to guarantee that 
\[ T_b \geq n + \frac{2(n+1)}{\sqrt{3}} + 3. \]
As $T_b \geq T_a$, it is enough to guarantee that 
\[ T_a + T_b = T_n - T_c + n \geq 2n + \frac{4(n+1)}{\sqrt{3}} + 6. \]

If $c \leq n-4$, then $T_n - T_c + n \geq 5n-6$. For $n \geq 37$,
$5n-6 > 2n + \frac{4(n+1)}{\sqrt{3}} + 6$, so $T_c-n$ does, in fact,
fit into the desired range.

It now remains to analyse the cases when $c > n-4$. As remarked earlier,
$c > n$ is impossible, and $c = n$ implies that $n = T_a + T_b$, so
two weighings suffice by Lemma~\ref{lem:sum-two-triangles}. So we are
left with three cases: $c = n-1$, $c = n-2$ and $c = n-3$. For such
$c$, $T_c$ is at least $T_{n-3}$, so
\[ T_a + T_b \leq 2n + (n-1) + (n-2) = 4n-3. \]
Therefore
$T_a \leq 2n - \frac{3}{2}$.  For $n \geq 21$, this implies $2a+3 \leq
n$. This fact allows us to use Corollary \ref{cor:big-range},
meaning that we have full use of the range $[a+1\ldots T_{n-1} -
  T_{a+1}]$.

Does $T_c - n$ fall into this range? For $n \geq 21$, the
lower bound follows from
\[ T_c - n \geq T_{n-3} - n \gg n > a + 1. \]

We prove the upper bound case by case.

\textbf{Case 1.} $c = n-3$.  We can rearrange the
upper bound condition
\[ T_{n-3} - n = T_c - n \leq T_{n-1} - T_{a+1} \]
to
\be T_a + a+1 \leq n + (n-1) + (n-2) = 3n - 3. \label{eqn:want-n-3} \ee
For $n$ this large, $2a+3 \leq n$ generously
implies
\[ a + 1 \leq n - \frac{3}{2}, \]
which, together with the known $T_a \leq 2n - \frac{3}{2}$, implies
(\ref{eqn:want-n-3}), so $S$ exists and the Baron succeeds.

\textbf{Case 2.} $c = n-2$.  Then $T_a + T_b = 3n-1$; therefore $T_a
\leq \frac{3n-1}{2}$. For the upper bound, we want
\[ T_{n-2} - n = T_c - n \leq T_{n-1} - T_{a+1}, \]
which rearranges to
\[ T_a + a+1 \leq n + (n-1). \]
Since we know $T_a \leq \frac{3n-1}{2}$, it suffices that 
\[ a+1 \leq \frac{n-1}{2}, \]
which is mercifully equivalent to the already established
condition $2a+3 \leq n$.  Therefore, the desired subset $S$
exists and the Baron succeeds.

\textbf{Case 3.} $c = n-1$.  Then $T_a + T_b = 2n$; therefore $T_a \leq
n$.  If $b = a$, then $T_a = n$ and the Baron succeeds in one weighing.
If $b = a+1$, then the Baron succeeds in two weighings by
Lemma~\ref{lem:two-close-triangles}.  

Now let us assume that $b \geq a+2$.  Therefore, $T_b \geq T_{a+2} >
T_a + (a+1) + (a+1)$.  Therefore 
\beas
2n = T_a + T_b & > & 2(T_a + (a+1)) \\
T_{a+1} & < & n \\
0 & < & n - T_{a+1} \\
T_c & < & n + T_{n-1} - T_{a+1} \\
T_c - n & < & T_{n-1} - T_{a+1},
\eeas
so $T_c - n$ fits in the desired range, the desired
subset $S$ exists, and the Baron succeeds.

The argument above proves that the Baron can convince his audience of
the weight of the $n$-gram coin among $n$ coins for $n \geq 37$.  The
theorem is completed by noting that
Lemma~\ref{lem:sum-two-triangles} covers all smaller $n$.

\section{Discussion}\label{sec:discussion}

Is it surprising that the answer is two?  That no matter how many
coins there are, Baron M\"{u}nchhausen can always prove the weight of
one of them in just two weighings on the scale?  It surprised us and
it surprised most people we gave this puzzle to.  At first, everyone
expects that this sequence should tend to infinity, or at least grow
without bound.

So we were most intrigued when we proved
Theorem~\ref{maximum} and discovered that the problem always has a
simple solution in three weighings. On reflection, however, maybe that discovery should have been less of
a surprise.  In standard coin-weighing puzzles, the person
constructing the weighings is trying to find something out; so they
are limited if nothing else by information-theoretic considerations,
and as the number of coins involved increases, the problem usually
becomes unequivocally more difficult.  In this puzzle, however, Baron
M\"{u}nchhausen has complete information.  So on the one hand, as the
number of coins increases, the audience knows less and the Baron's
task becomes more difficult; but on the other hand, the available
resources for constructing interesting weighings also grow.  In this
case, it turns out that these forces balance to produce a bounded
sequence.

In fact, demonstrating the weight of one coin among $n$ in two
weighings is \emph{easy}.  Think about what the Baron actually needs
to do to satisfy the conditions outlined in the beginning of the
proof, in Section~\ref{sec:main-proof}.  He must find some
decomposition of $n + T_n$ into three triangular numbers, and find
some subset of a certain collection of his coins that adds up to some
number.  The proof is long and hairy because we are trying to prove
that this subset \emph{always} exists, but the vast majority of the
time this is trivial.  How many ways are there to pick a subset of
integers from fifty to a hundred, so that their sum will be three
thousand?  Or three thousand one?  Gazillions!\footnote{Yes,
  ``gazillions'' is a technical term in advanced combinatorics.}  If
the range one has to work with is reasonably large, and the target sum
is comfortably between zero and the total sum of all the integers in
one's range, then of course one can find a subset, or a hundred
subsets.

Even more, any number $n$ generally has \emph{many} decompositions into
a sum of three triangular numbers---on the order of the square root of $n$ \cite{Grosswald}.  The proof in
Section~\ref{sec:main-proof} is hairy also because we were proving
that an \emph{arbitrary} decomposition of $n + T_n$ into three triangular
numbers leads to a solution, but in practice the Baron has the freedom
to pick and choose among a great number of possible decompositions.

The ease of proving one coin in two weighings suggests two future
directions.  One can explore how many ways there are for Baron
M\"{u}nchhausen to prove himself right.  One can also explore harder
tasks that can be asked of him.

For the outermost example,
we can ask the Baron to prove the weight of \emph{all} the given coins. For
$n=6$ this task will match the following puzzle, authored by Sergey
Tokarev \cite{tokarev}, that appeared at the last round of the Moscow Math
Olympiad in 1991:

\begin{quote}
    You have 6 coins weighing 1, 2, 3, 4, 5 and 6 grams that look the
    same, except for their labels. The number (1, 2, 3, 4, 5, 6) on
    the top of each coin should correspond to its weight. How can you
    determine whether all the numbers are correct, using the balance
    scale only twice?
\end{quote}

This task is clearly harder, and indeed this sequence does tend to
infinity: If the total number of coins is $n$, then the needed number
of weighings is always greater than $\log_3n$ \cite{blog1}. And again,
we give it as homework for the reader to prove this lower bound as well as the upper bound of $n-1$ weighings.

There is also a huge spectrum of possible intermediate tasks.  For
example, how many coins can the Baron show at once with at most two
weighings?  What is the smallest number of weighings the Baron needs
to specify two coins?  Or, given the total number of coins, how many
weighings does the Baron need to show the weight of a particular coin?
What if the audience can choose which coin's weight the Baron must
prove?  Which of these tasks can be done in a fixed maximum number of
weighings, and which can not?  What asymptotic behaviors of the number
of needed weighings occur?  What happens if we start using different
families of sets of available coins, not just $1\ldots n$? There is
plenty to be curious about!

\section{Acknowledgements}

We are grateful to Peter Sarnak, who suggested a potential way out
when we were stuck in our proof, even though we ended up not using his
suggestion, having shortly thereafter found another door to go through.

\bigskip
\hrule
\bigskip

\noindent 2000 {\it Mathematics Subject Classification}: Primary
11B99; Secondary 00A08, 11P99.

\noindent \emph{Keywords: } weighing, puzzle, triangular decomposition.

\bigskip
\hrule
\bigskip

\noindent
(Mentions A020756.)

\bigskip
\hrule
\bigskip


\begin{thebibliography}{9}

\bibitem{Guy} R.~K.~Guy and R.~J.~Nowakowsy, Coin-Weighing Problems, \emph{Amer. Math. Monthly} \textbf{102} (1995), 164-167.

\bibitem{BaronOriginal} S. Tokarev, XXVI all-Russian Mathematical Olympiad, Kvant, 2000, volume 5, pages 49-53 (in Russian)
\url{http://kvant.mccme.ru/pdf/2000/05/49.pdf}

\bibitem{OEIS} N. J. A. Sloane, Online Encyclopedia of Integer Sequences (OEIS). \url{http://www.research.att.com/~njas/sequences/}

\bibitem{Gauss} C.~F.~Gauss, \emph{Disquisitiones Arithmeticae}, Yale University Press. 1965.

\bibitem{Grosswald} E.~Grosswald, \emph{Representations of Integers as Sums of Squares}, Springer-Verlag. 1985.

\bibitem{Dunnington} G.~Waldo Dunnington, \emph{Carl Friedrich Gauss: Titan of Science}, The Mathematical Association of America, 2004

\bibitem{blog1} T. Khovanova, Coins Sequence, \url{http://blog.tanyakhovanova.com/?p=148}, 2009

\bibitem{blog2} T. Khovanova, A. Radul, Another Coins Sequence, \url{http://blog.tanyakhovanova.com/?p=179}, 2009

\bibitem{tokarev} Problems from the last round of LIV Moscow Mathematical Olympiad, \emph{Kvant}, \textbf{9} (1991), 70-71 (in Russian), or at:
\url{http://kvant.mirror1.mccme.ru/1991/09/zadachi_zaklyuchitelnogo_tura.htm}

\end{thebibliography}
\end{document}